\documentclass{amsart}

\usepackage{ulem}
\usepackage{multicol,latexsym}
\usepackage{textcomp}
\usepackage[ansinew]{inputenc}
\usepackage{amssymb}
\usepackage{bigfoot}
\usepackage{graphicx}
\usepackage{bm}
\usepackage{color}
\usepackage{natbib}
\usepackage{mathdots}
\usepackage{multicol}
\usepackage{amssymb, amsmath, amsthm}
\usepackage{bm}
\theoremstyle{plain}
\usepackage{graphicx}
\usepackage[aboveskip=-5pt,position=bottom]{caption}
\usepackage{url}
\usepackage{hyperref}
\usepackage{enumitem}

\newtheorem{lemma}{Lemma}[section]

\newtheorem{theorem}[lemma]{Theorem}

\newtheorem{exam}[lemma]{\normalfont \scshape
 Example}

\newcommand{\R}{\mathbb{R}}
\newcommand{\N}{\mathbb{N}}

\newcommand{\norm}[1]{\left\Vert#1\right\Vert}
\newcommand{\abs}[1]{\left\vert#1\right\vert}
\newcommand{\set}[1]{\left\{#1\right\}}

\newcommand{\bfzero}{\bm{0}}

\newcommand{\party}{\frac{\partial}{\partial y}C(x,y)}

\allowdisplaybreaks[4]
\begin{document}

\title[Independence of Bivariate Order Statistics]{Asymptotic Independence of Bivariate Order Statistics}
\author{Michael Falk and Florian Wisheckel}
\address{University of W\"{u}rzburg,
Institute of Mathematics,  Emil-Fischer-Str. 30, 97074 W\"{u}rzburg, Germany.}
\email{michael.falk@uni-wuerzburg.de, florian.wisheckel@uni-wuerzburg.de}



\begin{abstract}
It is well known that an extreme order statistic and a central order statistic (os) as well as an intermediate os and a central os from a sample of iid univariate random variables get asymptotically independent as the sample size increases. We extend this result to bivariate random variables, where the os are taken componentwise. An explicit representation of the conditional distribution of bivariate os turns out to be a powerful tool.
\end{abstract}

\maketitle

\section{Introduction}

Let $U_1,\dots,U_n$ be independent copies of a univariate random variable (rv) $U$ and denote by $U_{1:n}\le U_{2:n}\le \dots \le U_{n:n}$ the pertaining order statistics (os). It follows from Theorem 1.3 in \citet{falkre88} that there exists a universal constant such that for $1\le r\le n-k+1\le n$ and $n\in\N$
\begin{align}\label{eqn:asymptotic_independence_of_univariate_rv}
&\sup_{x,y\in\R}\abs{P(U_{r:n}\le x,U_{n-k+1:n}\le y)- P(U_{r:n}\le x)P(U_{n-k+1:n}\le y)}\\
&\le \mbox{const} \left(\frac{rk}{n(n-r-k+1)}\right)^{1/2}.\nonumber
\end{align}
This upper bound converges to $0$ if we consider a sequence $r = r(n)$ that satisfies $r/n\to_{n\to\infty}\lambda\in(0,1)$ together with $k=k(n)\to_{n\to\infty}\infty$, $k/n\to_{n\to\infty}0$. Then $(U_{r:n})$ is a sequence of central os, $(U_{n-k+1:n})$ a sequence of intermediate os and the limiting $0$ shows that they become asymptotically independent. The same holds for an intermediate sequence $r = r(n)$ together with fixed $k$, i.e., extreme os.

Starting with the work by \citet{gumbel46} on extremes, the asymptotic independence of order statistics has been investigated in quite a few articles. For detailed references we refer to \citet[p. 150]{gal87} and to \citet{falkko86}.

By the quantile transformation theorem (see, e.g. \citet[Lemma 1.2.4]{reiss89}) we can assume without loss of generality in the preceding result \eqref{eqn:asymptotic_independence_of_univariate_rv} that $U$ follows the uniform distribution on $(0,1)$.

Let $(U_1,V_1),\dots,(U_n,V_n)$ be independent copies of the bivariate rv $(U,V)$ that follows a copula, $C$ say, i.e., $U$ and $V$ are both uniformly distributed on $(0,1)$. Choose $r,k\in\set{1,\dots,n}$ and consider the vector $(U_{r:n},V_{k:n})$ of componentwise  os, called \emph{bivariate os}. In this paper we investigate the problem, whether asymptotic independence also holds for $(U_{r:n},V_{k:n})$ with proper sequences $r=r(n)$, $k=k(n)$.

Note that, for example, $U_{r:n}$ and $U_{n-r+1:n}$ with $r$ fixed get by inequality \eqref{eqn:asymptotic_independence_of_univariate_rv} asymptotically independent, but $U_{r:n}$ and $V_{n-r+1:n}$ might not. Consider $(U,V):= (U,1-U)$. Then the joint distribution of $(U,V)$ is a copula as well but $V_{n-r+1:n}=1-U_{r:n}$.

For $r=k=n$, the asymptotic joint distribution of $(U_{r:n},V_{k:n})$ is provided by multivariate extreme value theory. Precisely, if $n(U_{r:n}-1,V_{k:n}-1)$ has a non degenerate limit distribution $G$, say, then this limit has the representation
\[
G(x,y)=\exp\left(-\norm{(x,y)}_D\right),\qquad x,y\le 0,
\]
where $\norm\cdot_D$ is a particular norm on $\R^2$, called \emph{$D$-norm}, see, e.g., \citet[Section 4.4]{fahure10}. Current articles include \citet{aulfahozo14}, \citet{aulfazo14} and \citet{falk15}.

The limit distribution of $n(U_{n-i+1:n}-1, V_{n-j+1:n}-1)$ with fixed $i,j$ was established by \citet{gal75}. The set of limiting distributions in the intermediate case $(U_{n-k:n},V_{n-r:n})$ with $k=k(n)$, $r=r(n)$ both converging to infinity as $n$ increases, but $(k+r)/n\to_{n\to\infty}0$, was identified by \citet{chedehya97}. If in particular $n(U_{n:n}-1,V_{n:n}-1)$ converges in distribution to $G$ as above, then $(n/\sqrt k)\left(U_{n-k:n}-(n-k)/(n+1), V_{n-k:n}-(n-k)/(n+1)\right)$ follows asymptotically the bivariate normal distribution with mean vector $\bfzero\in\R^2$ and covariance matrix $\left(\begin{smallmatrix}1&2-\norm{(1,1)}_D\\2-\norm{(1,1)}_D&1\end{smallmatrix}\right)$ as shown by
\citet{falkwi16}. Asymptotic normality of $(U_{r:n},V_{k:n})$ in the central case, where $r/n\to_{n\to\infty}\lambda_1$, $k/n\to_{n\to\infty}\lambda_2$, $0<\lambda_1,\lambda_2<1$, is established in \citet{reiss89}.

In this paper we establish
\[
\sup_{x,y\in\R}\abs{P(U_{r:n}\le x,V_{k:n}\le y)- P(U_{r:n}\le x)P(V_{k:n}\le y)}\to_{n\to\infty}0,
\]
for various choices of $r=r(n)$ and $k=k(n)$. It turns out that for such sequences asymptotic independence holds with no further assumptions on the copula $C$. The main tool will be Lemma \ref{lem:conditional_expectation_of_multi_os_I}, in which the conditional distribution function (df) $P(U_{r:n}\le x\mid V_{k:n}=y)$ is derived for arbitrary $r,k\in\set{1,\dots,n}$. This powerful tool should be of interest of its own.

\section{Conditional Expectation of Bivariate OS}

In this section we compute as a major tool $P(U_{m:n}\le x\mid V_{k:n}=y)$ for arbitrary $m,k\in\set{1,\dots,n}$. For the formulation of Lemma \ref{lem:conditional_expectation_of_multi_os_I} and its proof it is quite convenient to explicitly quote Theorem 2.2.7 in \citet{nel06}.

\begin{theorem}[\citet{nel06}]\label{theo:bounds_for_partial_derivatives_of_copula}
Let $C$ be an arbitrary bivariate copula. For any $x\in[0,1]$, the partial derivative $\party$ exists for almost all $y$, and for such $x$ and $y$
\begin{equation}\label{eqn:bound_for_partial_derivative_of_copula}
0\le \party\le 1.
\end{equation}
Furthermore, the function $x\mapsto\party$ is defined and nondecreasing almost everywhere on $[0,1]$.
\end{theorem}

Now we are ready to state our major tool: we show that the conditional distribution $P(U_{m:n}\le x\mid V_{k:n}=y)$ is the linear combination of two probabilities concerning sums of independent Bernoulli rv.  We set, as usual, $U_{0:n}=V_{0:n}=0$ and $U_{n+1:n}=V_{n+1:n}=1$

\begin{lemma}\label{lem:conditional_expectation_of_multi_os_I}
Let $(U_1,V_1),\dots,(U_n,V_n)$, $n\in\N$, be independent copies of a rv $(U,V)$ that follows a copula $C$. Then we obtain for $1\le k,m\le n$ and for almost every $x,y\in[0,1]$
\begin{align}\label{eqn:representation_of_condition_expectation}
&P(U_{m:n}\le x\mid V_{k:n}=y)\nonumber\\
&= P\left(\sum_{i=1}^{k-1}1_{[0,x]}\left(U_i^{(1)}\right) + \sum_{i=1}^{n-k}1_{[0,x]}\left(U_i^{(2)}\right)\ge m\right)\nonumber\\
&\hspace*{1cm}+ \party P\left(\sum_{i=1}^{k-1}1_{[0,x]}\left(U_i^{(1)}\right) + \sum_{i=1}^{n-k}1_{[0,x]}\left(U_i^{(2)}\right)= m-1\right)
\end{align}
where $U_1^{(1)},\dots,U_k^{(1)},U_1^{(2)},\dots,U_{n-k}^{(2)}$ are independent rv with
\[
P\left(U_i^{(1)}\le u\right) = P(U\le u\mid V\le y) = \frac{C(u,y)}y
\]
and
\[
P\left(U_i^{(2)}\le u \right)  = P(U\le u\mid V> y) = \frac{u-C(u,y)}{1-y},\qquad 0\le u\le 1.
\]
\end{lemma}

If we choose, for example, $m=k=n$, then we obtain from the preceding result the representation
\begin{align*}
P(U_{n:n}\le x\mid V_{n:n}=y) &= \party P\left(\sum_{i=1}^{n-1}1_{[0,x]}\left(U_i^{(1)}\right)= n-1\right)\\
&= \party \frac{C(x,y)^{n-1}}{y^{n-1}}.
\end{align*}

\begin{proof}[Proof of Lemma \ref{lem:conditional_expectation_of_multi_os_I}]
We have
\begin{align*}
&P(U_{m:n}\le x\mid V_{k:n}=y)\\
&= \lim_{\varepsilon\downarrow 0}\frac{P(U_{m:n}\le x,\,V_{k:n}\in [y,y+\varepsilon])}{P(V_{k:n}\in[y,y+\varepsilon])}\\
&= \lim_{\varepsilon\downarrow 0} \frac{P(U_{m:n}\le x,\,V_{k:n}\le y+\varepsilon) - P(U_{m:n}\le x,\,V_{k:n}\le y)}{\varepsilon} \frac{\varepsilon}{P(V_{k:n}\in[y,y+\varepsilon])},
\end{align*}
where the second term on the right hand side above converges to $1/g_{k,n}(y)$ as $\varepsilon\downarrow 0$, where $g_{k,n}(\cdot)$ is the Lebesgue-density of $V_{k:n}$, see, e.g., \citet[Theorem 1.3.2]{reiss89}.

In the next step we will break the set $\set{U_{m:n}\le x,\,V_{k:n}\le y}$ into disjoint subsets. By $T$, $S$ we denote in what follows arbitrary subsets of $\set{1,\dots,n}$ and by $\abs T$, $\abs S$ their cardinalities, i.e., the numbers of their elements. Precisely, we have
\begin{align*}
&P(U_{m:n}\le x,\,V_{k:n}\le y)\\
&= P\left(\sum_{i=1}^n 1_{[0,x]}(U_i)\ge m,\, \sum_{i=1}^n 1_{[0,y]}(V_i)\ge k\right)\\
&= P\Biggl(\Bigl(\sum_{\abs T\ge m}\set{U_i\le x,i\in T;\,U_i>x,i\in T^\complement}\Bigr)\\
&\hspace*{2cm}\cap \Bigl(\sum_{\abs S\ge k}\set{V_i\le y,i\in S;\,V_i>y,i\in S^\complement}\Bigr)\Biggr)\\
&=\sum_{\abs T\ge m}\sum_{\abs S\ge k}  P\Biggl(\set{U_i\le x,i\in T;\,U_i>x,i\in T^\complement}\\
&\hspace*{2cm}\cap \set{V_i\le y,i\in S;\,V_i>y,i\in S^\complement}\Biggr)\\
&=\sum_{\abs T\ge m}\sum_{\abs S\ge k} P(U_i\le x,V_i\le y,\, i\in T\cap S) P\left(U_i\le x,V_i> y,\,i\in T\cap S^\complement\right)\\
&\hspace*{2cm}\times P\left(U_i>x,V_i\le y,\,i\in T^\complement \cap S\right) P\left(U_i>x,V_i>y,\,i\in T^\complement\cap S^\complement\right)\\
&=\sum_{\abs T\ge m}\sum_{\abs S\ge k} P(U\le x,V\le y)^{\abs{T\cap S}} P(U\le x, V>y)^{\abs{T\cap S^\complement}}\\
&\hspace*{3cm}\times P(U>x,V\le y)^{\abs{T^\complement\cap S}} P(U>x,V>y)^{\abs{T^\complement\cap S^\complement}}.
\end{align*}

As a consequence and by writing $x=\exp(\log(x))$ for $x\ge 0$ we obtain
\begin{align*}
&P(U_{m:n}\le x,\,V_{k:n}\le y+\varepsilon) - P(U_{m:n}\le x,\,V_{k:n}\le y)\\
&=\sum_{\abs T\ge m}\sum_{\abs S\ge k}\Biggl\{ \exp\Biggl(\abs{T\cap S}\log(P(U\le x,V\le y+\varepsilon))\\
&\hspace*{4cm} + \abs{T\cap S^\complement} \log(P(U\le x,V>y+\varepsilon))\\
&\hspace*{4cm} + \abs{T^\complement\cap S} \log(P(U> x,V\le y+\varepsilon))\\
&\hspace*{4cm} + \abs{T^\complement\cap S^\complement} \log(P(U> x,V>y+\varepsilon))\Biggr)\\
&\hspace*{3cm} - \exp\Biggl(\abs{T\cap S}\log(P(U\le x,V\le y )\\
&\hspace*{4cm} + \abs{T\cap S^\complement} \log(P(U\le x,V>y ))\\
&\hspace*{4cm} + \abs{T^\complement\cap S} \log(P(U> x,V\le y ))\\
&\hspace*{4cm} + \abs{T^\complement\cap S^\complement} \log(P(U> x,V>y ))\Biggr\}\\
&=\sum_{\abs T\ge m}\sum_{\abs S\ge k} \Biggl\{ \exp\Biggl( \abs{T\cap S} \log\Bigl(1 + \frac{P(U\le x, V\le y+\varepsilon)- P(U\le x,V\le y)}{P(U\le x,V\le y)}\Bigr)\\
&\hspace*{3cm} + \abs{T\cap S^\complement} \log\Bigl(1 + \frac{P(U\le x, V> y+\varepsilon)- P(U\le x,V > y)}{P(U\le x,V>y)}\Bigr)\\
&\hspace*{3cm} + \abs{T^\complement\cap S} \log\Bigl(1 + \frac{P(U> x, V\le y+\varepsilon)- P(U> x,V\le y)}{P(U> x,V\le y)}\Bigr)\\
&\hspace*{3cm} + \abs{T^\complement\cap S^\complement} \log\Bigl(1 + \frac{P(U> x, V> y+\varepsilon)- P(U> x,V> y)}{P(U> x,V> y)}\Bigr)\Biggr)\\
&\hspace*{4cm}- 1\Biggr\}\\
&\hspace*{4cm}\times P(U\le x,V\le y)^{\abs{T\cap S}} P(U\le x, V>y)^{\abs{T\cap S^\complement}}\\
&\hspace*{5cm}\times P(U>x,V\le y)^{\abs{T^\complement\cap S}} P(U>x,V>y)^{\abs{T^\complement\cap S^\complement}}
\end{align*}

We have for $\varepsilon\downarrow 0$ the expansions
\begin{align*}
P(U\le x,V\le y+\varepsilon)- P(U\le x,V\le y)&= \party \varepsilon + o(\varepsilon),\\
P(U\le x,V> y+\varepsilon)- P(U\le x,V> y)&= -\party \varepsilon + o(\varepsilon),\\
P(U> x,V\le y+\varepsilon)- P(U> x,V\le y)&=\left(1 -\party\right) \varepsilon + o(\varepsilon),\\
P(U> x,V> y+\varepsilon)- P(U> x,V> y)&=\left(\party-1\right) \varepsilon + o(\varepsilon).
\end{align*}

From the Taylor expansions $\log(1+x)=x+o(x)$ and $\exp(x)-1=x+o(x)$ as $x\to 0$ we, thus, obtain from the preceding equations
\begin{align*}
&\frac{P(U_{m:n}\le x,V_{k:n}\le y+\varepsilon) - P(U_{m:n}\le x,V_{k:n}\le y)}\varepsilon\\
&\to_{\varepsilon\downarrow 0}   \sum_{\abs T\ge m} \sum_{\abs S\ge k} \Biggl\{ \abs{T\cap S}\frac{\party}{p_1}\\
&\hspace*{4cm} - \abs{T\cap S^\complement} \frac{\party}{p_2}\\
&\hspace*{4cm} + \abs{T^\complement\cap S} \frac{1-\party}{p_3}\\
&\hspace*{4cm} + \abs{T^\complement\cap S^\complement} \frac{\party-1}{p_4}\Biggr\}\\
&\hspace*{4cm}\times p_1^{\abs{T\cap S}} p_2^{\abs{T\cap S^\complement}}
 p_3^{\abs{T^\complement\cap S}} p_4^{\abs{T^\complement\cap S^\complement}}\\
&=:  f(x,y)
\end{align*}
with
\[
p_1:=P((U,V)\in A_1):=P(U\le x,V\le y),\;p_2:=P((U,V)\in A_2):=P(U\le x,V>y),
\]
\[
p_3:=P((U,V)\in A_3):= P(U>x,V\le y),\; p_4:= P((U,V)\in A_4):= P(U>x,V>y).
\]
Note that $p_1+p_2+p_3+p_4=1$. Set
\[
n_j:=\sum_{i=1}^n 1_{A_j}(U_i,V_i),\qquad 1\le j\le 4.
\]
Then we obtain
\begin{align*}
f(x,y)&= E\Biggl(\Biggl\{ \frac{\party}{p_1}n_1 - \frac{\party}{p_2}n_2\\
&\hspace*{4cm} + \frac{1-\party}{p_3}n_3 + \frac{\party-1}{p_4}n_4\Biggr\}\\
&\hspace*{5cm}\times 1(n_1+n_2\ge m,n_1+n_3\ge k)\Biggr).
\end{align*}

Put, for notational convenience, $\xi_j:= (U_j,V_j)$, $1\le j\le n$. We have
\begin{align*}
&E(n_1 1(n_1+n_2\ge m, n_1+n_3\ge k))\\
&=\sum_{j=1}^n P\left(\set{\xi_j\in A_1}\cap\set{\sum_{i=1}^n1_{A_1\cup A_2}(\xi_i)\ge m, \sum_{i=1}^n 1_{A_1\cup A_3}(\xi_i)\ge k}\right)\\
&= n p_1 P\left(\sum_{i=1}^{n-1}1_{A_1\cup A_2}(\xi_i)\ge m-1, \sum_{i=1}^{n-1} 1_{A_1\cup A_3}(\xi_i)\ge k-1 \right)\\
&= n p_1 P\left(\sum_{i=1}^{n-1}1_{[0,x]}(U_i)\ge m-1, \sum_{i=1}^{n-1}1_{[0,y]}(V_i)\ge k-1 \right),
\end{align*}
as well as
\begin{align*}
&E(n_2 1(n_1+n_2\ge m, n_1+n_3\ge k))\\
&=n p_2 P\left(\sum_{i=1}^{n-1}1_{A_1\cup A_2}(\xi_i)\ge m-1, \sum_{i=1}^{n-1} 1_{A_1\cup A_3}(\xi_i)\ge k \right)\\
&= n p_2 P\left(\sum_{i=1}^{n-1}1_{[0,x]}(U_i)\ge m-1, \sum_{i=1}^{n-1}1_{[0,y]}(V_i)\ge k \right),
\end{align*}

\begin{align*}
&E(n_3 1(n_1+n_2\ge m, n_1+n_3\ge k))\\
&=n p_3 P\left(\sum_{i=1}^{n-1}1_{A_1\cup A_2}(\xi_i)\ge m, \sum_{i=1}^{n-1} 1_{A_1\cup A_3}(\xi_i)\ge k-1 \right)\\
&= n p_3 P\left(\sum_{i=1}^{n-1}1_{[0,x]}(U_i)\ge m, \sum_{i=1}^{n-1}1_{[0,y]}(V_i)\ge k-1 \right),
\end{align*}
and, finally,
\begin{align*}
&E(n_4 1(n_1+n_2\ge m, n_1+n_3\ge k))\\
&=n p_4 P\left(\sum_{i=1}^{n-1}1_{A_1\cup A_2}(\xi_i)\ge m, \sum_{i=1}^{n-1} 1_{A_1\cup A_3}(\xi_i)\ge k \right)\\
&= n p_4 P\left(\sum_{i=1}^{n-1}1_{[0,x]}(U_i)\ge m, \sum_{i=1}^{n-1}1_{[0,y]}(V_i)\ge k \right).
\end{align*}

Altogether we obtain from the preceding equations
\begin{align*}
f(x,y) &= n \party  P(U_{m-1:n-1}\le x, V_{k-1:n-1}\le y)\\
&\hspace*{.5cm} -n \party P(U_{m-1:n-1}\le x, V_{k:n-1}\le y)\\
&\hspace*{.5cm} +n\left(1-\party\right) P(U_{m:n-1}\le x, V_{k-1:n-1}\le y)\\
&\hspace*{.5cm} -n\left(1-\party\right) P(U_{m:n-1}\le x, V_{k:n-1}\le y)\\
&= n \party \Big( P(U_{m-1:n-1}\le x,\,V_{k-1:n-1}\le y)\\
&\hspace*{4cm}- P(U_{m-1:n-1}\le x,\,V_{k:n-1}\le y)\Big)\\
&\hspace*{0.5cm}+ n \left(1-\party\right)\Big(P(U_{m:n-1}\le x,\,V_{k-1:n-1}\le y)\\
&\hspace*{4cm}- P(U_{m:n-1}\le x,\,V_{k:n-1}\le y)\Big),
\end{align*}

We, thus, have established so far
\begin{align*}
&P(U_{m:n}\le x\mid V_{k:n}=y)\\
&= \frac{n}{g_{k,n}(y)}\Biggl\{\party P\left(U_{m-1:n-1}\le x,\sum_{i=1}^{n-1}1_{(0,y]}(V_i)=k-1\right)\\
&\hspace{1cm}+ \left(1-\party\right) P\left(U_{m:n-1}\le x,\, \sum_{i=1}^{n-1}1_{(0,y]}(V_i)=k-1\right)\Biggr\}\\
&=  \frac{n}{g_{k,n}(y)}\Biggl\{ P\left(\sum_{i=1}^{n-1}1_{[0,x]}(U_i)\ge m,\sum_{i=1}^{n-1}1_{(0,y]}(V_i)=k-1\right)\\
&\hspace{1cm}+  \party P\left(\sum_{i=1}^{n-1}1_{[0,x]}(U_i)= m-1,\sum_{i=1}^{n-1}1_{(0,y]}(V_i)=k-1\right)\Biggr\}.
\end{align*}

From the fact that
\[
P\left(\sum_{i=1}^{n-1}1_{[0,y]}(V_i)=k-1\right) = \frac{g_{k,n}(y)}n
\]
we, thus, obtain the representation
\begin{align*}
&P(U_{m:n}\le x\mid V_{k:n}=y)\\
&= P\left(\sum_{i=1}^{n-1}1_{[0,x]}(U_i)\ge m \mathrel{\Big|} \sum_{i=1}^{n-1} 1_{[0,y]}(V_i)=k-1\right)\nonumber\\
&\hspace*{1cm}+ \party P\left(\sum_{i=1}^{n-1}1_{[0,x]}(U_i)=m-1 \mathrel{\Big|} \sum_{i=1}^{n-1} 1_{[0,y]}(V_i)=k-1\right).
\end{align*}

We know from the theory of point processes (see, e.g. \citet[E.18]{reiss93}) that
\begin{align*}
&P\left(\sum_{i=1}^{n-1}1_{[0,x]}(U_i)\ge m \mathrel{\Big|} \sum_{i=1}^{n-1} 1_{[0,y]}(V_i)=k-1\right)\\
&=P\left(\sum_{i=1}^{n-1}1_{[0,x]}(U_i)\ge m \mathrel{\Big|} \sum_{i=1}^{n-1} 1_{[0,y]}(V_i)=k-1, \sum_{i=1}^{n-1} 1_{(y,1]}(V_i)=n-k\right)\\
&= P\left(\sum_{i=1}^{k-1}1_{[0,x]}\left(U_i^{(1)}\right) + \sum_{i=1}^{n-k} 1_{[0,x]}\left(U_i^{(2)}\right)\ge m\right),
\end{align*}
where $U_1^{(1)},\dots,U_{k-1}^{(1)},U_1^{(2)},\dots,U_{n-k}^{(2)}$ are independent rv with
\[
P\left(U_i^{(1)}\le u\right) = P(U\le u\mid V\le y) = \frac{C(u,y)}y
\]
and
\[
P\left(U_i^{(2)}\le u \right)  = P(U\le u\mid V> y) = \frac{u-C(u,y)}{1-y},\qquad 0\le u\le 1,
\]
which completes the proof of Lemma \ref{lem:conditional_expectation_of_multi_os_I}.
\end{proof}

\section{Asymptotic Independence of Order Statistics}

Throughout this section, $(U_{r:n},V_{k:n})$ denotes a rv of componentwise taken os pertaining to independent copies $(U_1,V_1),\dots,(U_n,V_n)$ of a rv $(U,V)$, which follows a copula $C$. By $X,Y,\eta_j$ we denote independent rv, where $X$ and $Y$ are standard normal distributed and $\eta_j$ has df $G_j(x)=\exp(x)\sum_{i=0}^{j-1}(-x)^i/i!$, $x\le 0$. The following main result establishes asymptotic independence of $U_{r:n}$ and $V_{k:n}$ for various sequences $r=r(n)$, $k=k(n)$, $n\in\N$.

\begin{theorem}
Let $k=k(n)$, $j=j(n)\in\set{1,\dots,n}$, $n\in\N$.
\begin{itemize}[leftmargin=.7cm]
\item[(i)] If $k$ satisfies $k\to_{n\to\infty}\infty$, $k/n\to_{n\to\infty}0$, then, for fixed $j\in\N$,
    \[
    \left(\frac n{\sqrt k}\left(U_{n-k+1:n}- \frac{n-k+1}{n+1}\right), n(V_{n-j+1:n}-1)\right) \to_D (X,\eta_j).
    \]
\item[(ii)] With $k$ and $j$ as in (i),
    \[
    \left(\frac n{\sqrt k}\left(U_{k:n}- \frac{k}{n+1}\right), n(V_{n-j+1:n}-1)\right) \to_D (X,\eta_j).
    \]
\item[(iii)]   If $k$ satisfies $k/n\to_{n\to\infty}\lambda\in(0,1)$ and $j\in\N$ is fixed, then
    \[
    \left(\sqrt n\left(U_{k:n} - \frac k{n+1}\right), n(V_{n-j+1:n}-1)\right) \to_D \left((\lambda(1-\lambda))^{1/2} X,\eta_j\right).
    \]
\item[(iv)]    With $k$ is chosen as in (iii) and $j\to_{n\to\infty}\infty$, $j/n\to_{n\to\infty}0$
    \begin{align*}
    &\left(\sqrt n\left(U_{k:n} - \frac k{n+1}\right), \frac n{\sqrt j}\left(V_{n-j+1:n}-\frac{n-j+1}{n+1}\right)\right)\\
     &\to_D \left((\lambda(1-\lambda))^{1/2} X,Y\right).
    \end{align*}
\item[(v)] With $k$ as chosen in (i), $j$ chosen as in (iv) and, in addition, $j/\sqrt  k\to_{n\to\infty}0$, 
    \[
    \left(\frac n{\sqrt k}\left(U_{n-k+1:n}-\frac{n-k+1}{n+1}\right), \frac n{\sqrt j}\left(V_{n-j+1:n}-\frac{n-j+1}{n+1}\right)  \right)\to_D (X,Y).
    \]
\end{itemize}
\end{theorem}

More results can immediately be deduced from the preceding result by noting that $(1-U_{r:n},1-V_{k:n})= \left(\bar U_{n-r+1:n},\bar V_{n-k+1:n}\right)$, which are os pertaining to the iid sequence $(\bar U_1,\bar V_1),\dots,(\bar U_n,\bar V_n)=(1-U_1,1-V_1),\dots,(1-U_n,1-V_n)$ with copula $\bar C(u,v)= P(1-U\le u,1-V\le v)$.

\begin{proof}
We prove only assertion (i). The remaining parts can be shown in complete analogy. By $P*X$ we denote in what follows the distribution of a rv $X$, i.e., $(P*X)(B)=P(X\in B)$ for any $B$ in the Borel-$\sigma$-field of $\R$. We have with $\mu_n:=(n-k+1)/(n+1)$ and $x\in\R$, $y<0$ by \eqref{eqn:representation_of_condition_expectation} the representation
\begin{align}\label{eqn:expansion_of_probability_via_cond_exp}
&P\left(\frac n{\sqrt k}(U_{n-k+1:n}-\mu_n)\le x, n(V_{n-j+1:n}-1)\le y\right)\nonumber\\
&= \int_{-n}^y P\left(\frac n{\sqrt k}(U_{n-k+1:n}-\mu_n)\le x \mid n (V_{n-j+1:n}-1)=z \right)\nonumber\\
&\hspace*{7cm} (P*n (V_{n-j+1:n}-1))(dz)\nonumber\\
&= \int_{-n}^yP\left(U_{n-k+1:n}\le \frac{\sqrt k}n x + \mu_n \mid V_{n-j+1:n} = 1 + \frac zn\right)\nonumber\\
&\hspace*{7cm} (P*n (V_{n-j+1:n}-1))(dz)\nonumber\\
&= \int_{-n}^y P\left(\sum_{i=1}^{n-j}1_{\left[0,\frac{\sqrt k}n x +\mu_n\right]}\left(U_i^{(1)}\right) + \sum_{i=1}^{j-1} 1_{\left[0,\frac{\sqrt k}n x +\mu_n\right]}\left(U_i^{(2)}\right) \ge n-k+1 \right)\nonumber\\
&\hspace*{1cm} + \frac{\partial}{\partial y}C(x,1+\frac{z}{n}) \nonumber\\
&\hspace*{2cm}\times P\left(\sum_{i=1}^{n-j}1_{\left[0,\frac{\sqrt k}n x +\mu_n\right]}\left(U_i^{(1)}\right) + \sum_{i=1}^{j-1} 1_{\left[0,\frac{\sqrt k}n x +\mu_n\right]}\left(U_i^{(2)}\right) = n-k \right)\nonumber\\
&\hspace*{7cm} (P*n (V_{n-j+1:n}-1))(dz).
\end{align}

It is well known that $n (V_{n-j+1:n}-1)\to_D G_j$, see, e.g. equation (5.1.28) in \citet{reiss89}.

We claim that
\begin{align*}
P\left(\sum_{i=1}^{n-j}1_{\left[0,\frac{\sqrt k}n x +\mu_n\right]}\left(U_i^{(1)}\right) + \sum_{i=1}^{j-1} 1_{\left[0,\frac{\sqrt k}n x +\mu_n\right]}\left(U_i^{(2)}\right) \ge n-k+1 \right) \to_{n\to\infty} \Phi(x),
\end{align*}
where $\Phi(\cdot)$ denotes the df of the standard normal distribution.

Note that
\[
p_n:= P\left(U_i^{(1)}\le \frac{\sqrt k}n x+\mu_n \right) = \frac{C\left( \frac{\sqrt k}n x+\mu_n, 1+\frac zn\right)}{1+\frac zn} \to_{n\to\infty} 1
\]
and
\begin{align*}
1-p_n&= \frac{1+\frac zn- C\left( \frac{\sqrt k}n x+\mu_n, 1+\frac zn\right)}{1+\frac zn}\\
&= \frac{1+\frac zn- \frac{\sqrt k}n x-\mu_n +\left(\frac{\sqrt k}n x+\mu_n - C\left( \frac{\sqrt k}n x+\mu_n, 1+\frac zn\right) \right)}{1+\frac zn}\\
&= \frac{\frac zn -\frac{\sqrt k}n x+ \frac k{n+1} +\int_{1+z/n}^1 \frac{\partial}{\partial v}C\left( \frac{\sqrt k}n x+\mu_n, v\right) \,dv}{1+\frac zn}\\
&= \frac{-\frac{\sqrt k}n x + \frac k{n+1} + O\left(\frac zn\right)}{1+\frac zn}
\end{align*}
by Theorem \ref{theo:bounds_for_partial_derivatives_of_copula}. We obtain that $(n-j)p_n(1-p_n)$ is of order $k(n)$ as $n\to\infty$ and, thus, the central limit theorem for arrays of Binomial distributions implies
\[
\frac{\sum_{i=1}^{n-j}\left(1_{\left[0,\frac{\sqrt k}n x+\mu_n\right]}\left(U_i^{(1)}\right) - p_n\right)  }{\left((n-j)p_n(1-p_n)\right)^{1/2}} \to_D N(0,1).
\]

As a consequence we obtain
\begin{align*}
&P\left(\sum_{i=1}^{n-j}1_{\left[0,\frac{\sqrt k}n x+\mu_n\right]}\left(U_i^{(1)}\right) + \sum_{i=1}^{j-1}1_{\left[0,\frac{\sqrt k}n x+\mu_n\right]}\left(U_i^{(2)}\right)\ge n-k+1\right)\\
&= P\left(\frac{\sum_{i=1}^{n-j}\left(1_{\left[0,\frac{\sqrt k}n x+\mu_n\right]}\left(U_i^{(1)}\right) - p_n\right)  }{\left((n-j)p_n(1-p_n)\right)^{1/2}} +o(1) \ge \frac{n-k+1 -(n-j)p_n}{\left((n-j)p_n(1-p_n)\right)^{1/2}}  \right)\\
&\to_{n\to\infty} 1-\Phi(-x) = \Phi(x),
\end{align*}
since
\[
\frac{n-k+1 -(n-j)p_n}{\left((n-j)p_n(1-p_n)\right)^{1/2}} = \frac{n(1-p_n)-k +O(1)}{\sqrt k (1+o(1))} = -x + o(1).
\]
This implies that the integrand in representation \eqref{eqn:expansion_of_probability_via_cond_exp} converges to $\Phi(x)$. The assertion now follows from the dominated convergence theorem.
\end{proof}

\section*{Acknowledgment} The authors are indebted to Professor Gennady Samorodnitsky for raising the problem investigated in this paper during the workshop \textit{extreme value and time series analysis}, 21-23 March 2016, Karlsruhe Institute of Technology, Germany.

\bibliographystyle{enbib_arXiv}
\bibliography{evt}

\end{document}